\documentclass{amsart}
\usepackage{graphicx} 
\usepackage{amsmath,dsfont}
\usepackage{amssymb}
\usepackage{xcolor} 
\usepackage{mathtools}
\usepackage{colortbl} 
\usepackage{tikz} 
\usepackage{hf-tikz} 
\usepackage{makecell}
\usepackage{mathrsfs}
\usepackage{graphicx}
\usepackage{enumerate}
\usepackage{chngcntr}
\usepackage{enumitem}
\usepackage[colorlinks,
linkcolor=red,
anchorcolor=blue,
citecolor=green
]{hyperref}

\newtheorem{theorem}{Theorem}
\newtheorem*{theorem*}{Theorem}
\numberwithin{theorem}{section}
\newtheorem{definition}{Definition}[section]
\newtheorem{prop}{Proposition}[section]
\newtheorem{lemma}{Lemma}[section]
\newtheorem{remark}{Remark}[section]
\newtheorem{conjecture}{Conjecture}[section]
\newtheorem{corollary}{Corollary}[section]

\newcommand{\R}{\mathbb{R}}

\DeclarePairedDelimiter\abs{\lvert}{\rvert}
\DeclarePairedDelimiter\norm{\lVert}{\rVert}
\DeclarePairedDelimiter\inner{\langle}{\rangle}

\makeatletter
\let\oldabs\abs
\def\abs{\@ifstar{\oldabs}{\oldabs*}}
\let\oldnorm\norm
\def\norm{\@ifstar{\oldnorm}{\oldnorm*}}
\let\oldinner\inner
\def\inner{\@ifstar{\oldinner}{\oldinner*}}
\makeatother

\def\T{\mathbb{T}} 

\numberwithin{prop}{section}
\counterwithout{table}{section}

\title
[Restriction and Kakeya maximal estimates in $\R^4$]
{Restriction and Kakeya maximal estimates in $\R^4$}

\author[T.Borges]{Tainara Borges}
\address{Tainara Borges, Department of Mathematics, University of Pennsylvania, Philadelphia, PA. 19104, USA}
\email{tborges@sas.upenn.edu}

\author[T.Chan]{Tiklung Chan}
\address{Tiklung Chan, Department of Mathematics, University of California, San Diego, La Jolla 92093-0112, USA}
\email{tic017@ucsd.edu}

\author[M.Chen]{Mingfeng Chen}
\address{Mingfeng Chen, Department of Mathematics, University of Wisconsin-Madison, Madison, WI. 53706, USA}
\email{mchen454@wisc.edu}

\author[D.Liu]{Diankun Liu}
\address{Diankun Liu, School of Mathematical Sciences, Zhejiang University, Hangzhou 310027, China}
\email{liudiankun5@gmail.com}

\author[Y.Xi]{Yakun Xi}
\address{Yakun Xi, School of Mathematical Sciences, Zhejiang University, Hangzhou 310027, China}
\email{yakunxi@zju.edu}

\author[Y.Zhan]{Yufei Zhan}
\address{Yufei Zhan, Department of Mathematics, University of Pennsylvania, Philadelphia, PA. 19104, USA}
\email{zhanyf@math.upenn.edu}

\begin{document}

\begin{abstract}
By combining the planebrush argument of Katz and Zahl \cite{katz21} with the decoupling-incidence method of Wang and Wu \cite{WangWu2024}, we derive new bounds for the Fourier restriction problem and the Bochner--Riesz problem, extending the range to $p > 2 + \frac{200}{251}$ in $\mathbb{R}^4$. Moreover, leveraging the two-ends Furstenberg estimate in the plane, we also obtain a Kakeya maximal estimate in $\mathbb{R}^4$ at dimension $3.054$.
\end{abstract}

\maketitle

\section{Introduction}

The Fourier extension operator is defined by
\begin{equation}\label{eq Extension}
    \mathcal E_{S}f(x,t)\coloneqq \int_{B_{1}^{n-1}} e^{2\pi i (x\cdot \xi + t\psi(\xi))}\, f(\xi)\, d\xi,
\end{equation}
where the hypersurface $S=\{(\xi,\psi(\xi)):\xi\in B_{1}^{n-1}\}\subset \mathbb{R}^{n}$ has strictly positive second fundamental form, and $B_{1}^{n-1}$ denotes the unit ball in $\mathbb{R}^{n-1}$.
 For convenience, we use $\mathcal Ef$ to denote $\mathcal E_{S}f$ when the hypersurface $S$ is clear from the context.
Stein proposed the following conjecture concerning this operator in \cite{stein79}:

\begin{conjecture}
\label{cnj: Restriction conjecture}
    Suppose that $S\subset \mathbb{R}^{n}$ is a compact $C^{2}$ hypersurface with a strictly positive second fundamental form. If $p>\frac{2n}{n-1}$, then there exists a constant $C_{p}>0$ such that 
    \begin{equation}
    \label{eq: restriction}
        \|\mathcal E_{S}f\|_{L^{p}(\mathbb{R}^{n})}\leq C_{p} \|f\|_{L^{p}(B_1^{n-1})}.
    \end{equation}
\end{conjecture}
The restriction problem concerns the uniform control of Fourier transforms of measures supported on surfaces. 
Heuristically, $\mathcal E_{S}f$ is a superposition of wave packets
that are localized on dual thin tubes in the physical space and the frequency space pointing in various directions (dual to their frequency localizations) and oscillating at different frequencies, which is called the 
\emph{wave packet decomposition}. Thus, improved understanding of wavepackets leads to progress on the restriction conjecture---indeed, all the modern approaches to the problem are based on wavepacket analysis.

When $n=2$, this conjecture was proven by Fefferman \cite{Fefferman} and for $n\geq 3$, the problem remains open. In $\mathbb{R}^{3}$, Guth \cite{guth2016restriction} introduced the polynomial method and proved the conjecture holds for $p>3+\frac{1}{4}$. Later, Wang \cite{Wangbrooms} proved it holds for $p>3+\frac{3}{13}$ using brooms.

Recently, Wang and Wu \cite{WangWu2024} introduced a new framework to study restriction estimates by combining 
incidence estimates for the wave packet tubes and oscillation controls, namely, refined decoupling, on the overlaps of the tubes. 
They proved a two-ends Furstenberg inequality in the plane, using the hairbrush argument to obtain an incidence estimate, and established the restriction estimate for $p > 22/7$ in three dimensions.
In higher dimensions, there are many results in the literature (see \cite{Guth18,HickmanRogers,HickmanZahl,GWZ,WangWu2024}).

In this paper, we study the restriction problem in $\mathbb{R}^{4}$. 
Guth \cite{Guth18} proved that the restriction estimate holds for $p>2+\frac{4}{5}=2.8$. 
Demeter \cite{Demeter} showed that it holds for $p>2+\frac{66642}{83303}\approx 2.79999$, 
and Hickman and Rogers \cite{HickmanRogers} proved that it holds for $p>2+\frac{1407}{1759}\approx 2.79989$.
Our main result improves the restriction estimate in $\mathbb{R}^{4}$ by combining Katz--Zahl's planebrush argument \cite{katz21} with the Wang--Wu decoupling-incidence method \cite{WangWu2024}.
\begin{theorem}\label{thm: restriction}
    For $p>2+\frac{200}{251}\approx 2.7968$, there exists a constant $C_{p}>0$ such that
    \begin{equation}
        \|\mathcal E_{S}f\|_{L^{p}(\mathbb{R}^{4})} \leq C_{p} \|f\|_{L^{p}(B_{1}^{n-1})}.
    \end{equation}
\end{theorem}

Moreover, as shown in \cite{GWX}, the decoupling-incidence argument extends to general H\"ormander operators and, in particular, to the Bochner--Riesz operator:
\[
m^{\delta}(D)f := \int_{\mathbb{R}^{n}} e^{i x \cdot \xi} \, m^{\delta}(\xi) \, \hat{f}(\xi) \, d\xi, \quad m^{\delta}(\xi) = (1 - |\xi|^{2})_{+}^{\delta}.
\]
Therefore, our improved incidence estimate also implies the following new bounds for the Bochner--Riesz multiplier in $\R^4$, which improve upon the previously best-known range $\max\{p,p'\}>\frac{14}{5}$ due to Guo, Oh, Wang, Wu and Zhang \cite{GOWWZ}.
See Remark \ref{rem BR} for a detailed discussion.

\begin{corollary}\label{cor BochnerRiesz}
     The Bochner--Riesz operator $ m^{\delta}(D) $ is bounded on $ L^{p}(\mathbb{R}^{4}) $ in the optimal range  
\[
\delta > 4\Big|\frac{1}{2} - \frac{1}{p}\Big| - \frac{1}{2}
\]  
whenever $ \max\{p, p'\} \geq 2 + \frac{200}{251} $, where $ p' $ is the Hölder conjugate of $ p $.
\end{corollary}

The wave packets 
contribute efficiently to a function with large $ L^p $-norm
when their tubular supports exhibit significant overlap. Therefore, the proofs of Theorem \ref{thm: restriction} and Corollary \ref{cor BochnerRiesz} implicitly rely on the fact that collections of thin tubes pointing in different directions cannot have excessively large intersections.
In other words, a putative counterexample to a Kakeya conjecture could be leveraged to construct a function $ f $ violating the restriction inequality. 

This motivates the study of these geometric objects, known as Kakeya sets.

\begin{definition}
A set $K$ is called a \emph{Kakeya set} in $\mathbb{R}^{n}$ if, for every $v\in \mathbb{S}^{n-1}$, it contains a unit line segment with direction $v$.
\end{definition}

These sets are central to a fundamental open problem in geometric measure theory.
\begin{conjecture}\label{cnj: Kakeya set conjecture}
    Every Kakeya set has Hausdorff dimension $n$ in $\mathbb{R}^{n}$.
\end{conjecture}
The conjecture was first settled in two dimensions by Davies \cite{Davies}, and a major recent advance is its resolution in three dimensions due to Wang and Zahl \cite{Wangzahl}.
There exists a stronger quantitative version of this problem known as the Kakeya maximal conjecture:

\begin{conjecture}[Kakeya maximal estimate at dimension $d$]
\label{cnj: Kakeya Maximal conjecture}
    Let $\mathbb{T}$ be a set of $\delta$-tubes in $\mathbb{R}^{n}$ that point in $\delta$-separated directions. Then, for each $d\leq n$ and $\varepsilon >0$, there exists a constant 
    $C_{\varepsilon}>0$ such that
    \begin{equation}
    \label{eq: Kakeya Maximal}
        \Big\|\sum_{T\in\mathbb{T}}\chi_{T}\Big\|_{\frac{d}{d-1}}\leq C_{\varepsilon} 
        \left(
            \frac{1}{\delta}
        \right)^{\frac{n}{d}-1+\varepsilon}.
    \end{equation}
\end{conjecture}

We briefly recall the history of the Kakeya problem in $\mathbb{R}^{4}$. Drury \cite{Drury} established that every Kakeya set in $\mathbb{R}^{4}$ has Hausdorff dimension at least $\frac{5}{2}$. Subsequently, Wolff \cite{Wolff95} proved the Kakeya maximal estimate at dimension $3$ using the hairbrush argument, which then implies that Kakeya sets have Hausdorff dimension at least $3$. This is an important bound in the four-dimensional Kakeya maximal problem that is difficult to improve. Later, Guth and Zahl \cite{Guth-Zahl, ZahlJoshua} established the Kakeya maximal estimate at dimension $3+\frac{1}{40}=3.025$ using the polynomial Wolff axioms. Katz and Zahl \cite{katz21} refined the estimate for the plany case by employing the planebrush argument, thereby establishing the Kakeya maximal estimate at dimension $3.049$ and demonstrating a lower bound of $3.059$ for the Hausdorff dimension of Kakeya sets. 

By combining the planebrush argument of Katz--Zahl with the two-ends Furstenberg inequality from \cite{WangWu2024}, we improve the Kakeya maximal estimate in $\R^4$ to dimension $3.054$.

\begin{theorem} 
\label{thm: Kakeya Maximal}
Let $\mathbb{T}$ be a set of $\delta$-tubes in $\mathbb{R}^{4}$ that point in $\delta$-separated directions. Then for each $\varepsilon>0$, there exists a constant $C_{\varepsilon}$ such that 
\begin{equation}
        \Big\|\sum_{T\in\mathbb{T}}\chi_{T}\Big\|_{\frac{d_0}{d_0-1}}\leq C_{\varepsilon} 
        \left(
            \frac{1}{\delta}
        \right)^{\frac{4}{d_0}-1+\varepsilon}, 
    \end{equation}
where $d_0=\frac{159 + \sqrt{145}}{56}\approx 3.0543$.    
\end{theorem}

By testing the restriction estimate on non-oscillatory wave packets, one sees that Theorem \ref{thm: restriction} already implies a Kakeya maximal estimate, but only at dimension $3.02$.   
On the other hand, the Kakeya maximal problem does not involve oscillatory behavior on the tubes to begin with. By dispensing with the compatibility required in the decoupling-incidence regime, we obtain that Theorem \ref{thm: Kakeya Maximal} achieves a Kakeya maximal estimate at a slightly higher dimension.

\subsection*{A Discussion of Numerologies}
Both the restriction and Kakeya maximal improvements, Theorem \ref{thm: restriction} and Theorem \ref{thm: Kakeya Maximal}, rely on improved incidence estimates for a family of two-ends shaded tubes in $\R^4$. We discuss only numerology here with the following notation; detailed definitions are deferred until Section \ref{sec: prelim}.

\begin{definition}
    
Let $0<\varepsilon_{2}<1$. We say the \emph{a priori} incidence estimate $\operatorname{\textbf{\bf TE}}(d,a,b)$ holds if for any $\varepsilon>0$, there exists a constant $c=c(d,a,b, \varepsilon, \varepsilon_2)$ such that 
\begin{align}
    \left|
        \bigcup_{T\in\T}
        Y(T)
    \right|
    \geq 
    c \lambda^a
    \delta^{4-d+\varepsilon}
    (\delta^3\#\T)^b
\label{eq: TE(d,a,b) intro}
\end{align}
holds for any family $\T$ of $\delta$-tubes pointing in $\delta$-separated directions that have an associated $\lambda$-dense, $(\varepsilon_1, \varepsilon_{1}\varepsilon_2)$-two-ends\footnote{The two-ends conditions 
used in \cite{WangWu2024, katz21}
have some technical distinctions.
We adopt the weaker $(\varepsilon_1, \varepsilon_2, C)$-two-ends condition as in the former and drop the multiplicative constant $C$ for convenience, where $\varepsilon_1$ specifies the size of testing sub-tubes. See Definition \ref{def: two-ends}. 
} 
shading $Y(T)$ for any $0<\varepsilon_{1}<1$. 
\end{definition}

\begin{itemize}
    \item For each $d\leq n$ and some $b\geq0$, the incidence estimate
\begin{align}
\label{eq: TE(d,d+1/2,b}
    \operatorname{{\bf TE}}
    \left(
        d,\tfrac{d+1}{2},b
    \right) 
\end{align}
combined with refined decoupling would imply that Conjecture \ref{cnj: Restriction conjecture} in $\R^n$ is true for 
$p>\frac{2(d+n)}{d+n-2}$, following the decoupling--incidence method in \cite{WangWu2024}.

\item For each $d\leq n$ and some $b\geq0$, the weaker incidence estimate 
\begin{align}
\label{eq: TE(d,d,b)}
    \operatorname{{\bf TE}}\left(
        d,d,b
    \right)
\end{align}
would imply, by the standard two-ends reduction
(e.g. \cite{Wolff95,Tao2009TwoEnds,katz21}), 
that Conjecture \ref{cnj: Kakeya Maximal conjecture} in $\R^n$ is true at dimension $d$.

\item  For each $d\leq n$ and some $b, C\geq0$, the even weaker incidence estimate 
\begin{align}
\label{eq: TE(d,C,b)}
    \operatorname{{\bf TE}}\left(
        d,C,b
    \right)
\end{align}
would imply a partial result toward Conjecture \ref{cnj: Kakeya set conjecture}, namely that every Kakeya set in $\R^n$ has Hausdorff dimension at least $d$. 

\end{itemize}

For each $d\leq n$, the incidence estimate 
\eqref{eq: TE(d,d+1/2,b}
would imply 
\eqref{eq: TE(d,d,b)}, 
which would in turn imply 
\eqref{eq: TE(d,C,b)}. 
This aligns with the fact that in $\R^n$, the restriction estimate
(\ref{eq: restriction}) for all $p>\frac{2(d+n)}{d+n-2}$ 
implies the Kakeya maximal estimate 
(\ref{eq: Kakeya Maximal}) at dimension $d$, 
which in turn implies that every Kakeya set in $\R^n$ has Hausdorff dimension at least $d$.

For the four-dimensional case in this paper, we prove two distinct incidence estimates. Lemma \ref{lem incidence} provides an estimate of the form \eqref{eq: TE(d,d+1/2,b}, which features a smaller exponent on $\lambda$  suitable for the restriction problem. Lemma \ref{lem incidence2} is of the form \eqref{eq: TE(d,d,b)}, which 
allows a higher power of $\lambda$ in return for an improvement in the dimension $d$.

Finally, we note that the exponent $b$ becomes relevant if one wishes to study $L^q \to L^p$ estimates away from the diagonal, as was done in \cite{GWX} in various cases. In our situation, however, it seems difficult to obtain an incidence estimate with a favorable exponent $b$. Therefore, in this paper, we restrict ourselves to the diagonal case.

\subsection*{Acknowledgements}
This work originated at the Research Groups in Analysis 2025 workshop, in which the authors Tainara Borges, Tiklung Chan, Mingfeng Chen, and Yufei Zhan participated. These four authors acknowledge support from the U.S. National Science Foundation and the Department of Mathematics at the University of Pennsylvania, and thank the organizers of the workshop for providing an excellent research environment. They are deeply grateful to Xiumin Du and Shukun Wu for their insightful guidance and mentorship during and after the workshop. Diankun Liu and Yakun Xi are partially supported by the National Key R\&D Program of China Grant No. 2022YFA1007200, NSFC Grant No. 12571107, and the Zhejiang Provincial NSFC Grant No. LR25A010001. The authors also thank Joshua Zahl for valuable discussions and helpful correspondence.

\subsection*{Notation} 
\begin{enumerate}
    \item We fix $\varepsilon_{0} = \varepsilon^{1000}$.
    \item We write $a \lessapprox b$ to mean that for any $\varepsilon > 0$ and a large parameter $P \geq 1$ which should be clear from context (typically $R$ or $\delta^{-1}$), there exists a constant $C_{\varepsilon}$ such that $a \leq C_{\varepsilon} P^{\varepsilon} b$.
    \item We write $a \gtrapprox b$ to mean that for any $\varepsilon > 0$ and a large parameter $P \geq 1$ which should be clear from context, there exists a constant $C_{\varepsilon}$ such that $a \geq C_{\varepsilon} P^{-\varepsilon} b$.
    \item We use $ B^m_r(x)$ to denote an $m$-dimensional ball of radius $r$ centered at $x$; we may simply write $ B^m_r$ or $ B_r$ if $x$ and $m$ are clear from context.
\end{enumerate}

\section{Preliminaries}\label{sec: prelim}

\subsection{Incidence estimates} In this section, we recall the definitions and incidence estimates that we will use later. 

\begin{definition}[Refinement]
    For two finite sets $E,F$, we say $E$ is a refinement of $F$ if $E\subset F$ and $\#E\gtrapprox \#F$.
\end{definition}

\begin{definition}[Shading]
    Let $T$ be a $\delta$-tube. A shading of $T$ is a set $Y(T)$ that is a union of $\delta$-cubes contained in $T$. Let $\mathbb{T}$ be a family of $\delta$-tubes; for each $T\in\mathbb{T}$, let $Y(T)$ be a shading of $T$. We denote the pair $(\mathbb{T},Y)$ as a family of tubes and their associated shading.
\end{definition}

Next, we define some fundamental quantities associated with the pair $(\mathbb{T},Y)$. 
\begin{definition}
For each $\delta$-cube $Q\in E_{Y}\coloneqq\bigcup_{T\in\mathbb{T}}Y(T)$,
we define $$\mathbb{T}_{Y}(Q)\coloneqq\{T\in\mathbb{T}:Q\subset Y(T)\}.$$
\end{definition}

\begin{definition}
    We define the average shading density and average multiplicity of the family of tubes with shading $(\mathbb{T},Y)$ by
    \begin{equation}\label{eq average dense}
        \lambda_{Y}\coloneqq \frac{1}{\#\mathbb{T}}\sum_{T\in\mathbb{T}}\frac{|Y(T)|}{|T|}.
    \end{equation}
\begin{equation}\label{eq average multiplicity}
    \mu_{Y}\coloneqq \frac{1}{|E_{Y}|}\sum_{Q\in E_{Y}}|\mathbb{T}_{Y}(Q)|.
\end{equation}
\end{definition}

We are particularly concerned with a family of tubes with associated shading that possess a special structure, which was proposed in \cite{Guth-Zahl} and  \cite{katz21}. Roughly speaking, since the directions passing through a point are concentrated along a curve in $\mathbb{S}^{3}$, it can be regarded as a narrow part.

\begin{definition}[Plany] Let $(\mathbb{T},Y)$ be a set of $\delta$-tubes and their associated shading. We say that $(\mathbb{T},Y)$ is plany if for each $\delta$-cube $Q$, there is a plane
$\Pi(Q)$ so that for all $T\in \mathbb{T}_{Y}(Q)$ we have $\angle (V(T),\Pi(Q))\lesssim \delta$, where $V(T)$ denotes the direction of $T$.
\end{definition}

Recall the definition of the two-ends condition, which reflects the shading of tubes with the non-concentration property. 
\begin{definition}[Two-ends]
\label{def: two-ends}
Let $\delta\in (0,1)$ and let $(\mathbb{T},Y)$ be a $\delta$-separated tube family with associated shading. Let $0<\varepsilon_{2}<\varepsilon_{1}<1$. We say the shading
$Y$ is $(\varepsilon_{1},\varepsilon_{2})$-two-ends if for all $T\in \mathbb{T}$ and all $\delta\times \delta^{\varepsilon_{1}}$-tubes $J\subset T$,
$$|Y(T)\cap J|\lesssim \delta^{\varepsilon_{2}}|Y(T)|.$$  
\end{definition}

The above definition tells us that the shading of each tube cannot concentrate too much on a $\delta \times \delta^{\varepsilon_{1}}$-subtube. Note that if $\varepsilon_{2} < \varepsilon_{1} < \varepsilon_{1}'$, then $(\varepsilon_{1}, \varepsilon_{2})$-two-ends implies $(\varepsilon_{1}', \varepsilon_{2})$-two-ends.

The following condition can be regarded as the ``two-ends"  condition of the direction set.
\begin{definition}[Robust transversality]
    Let $\delta\in (0,1)$ and let $(\mathbb{T},Y)$ be a $\delta$-separated tube family with associated shading. Let $0<\varepsilon_{1}<1$. We say $Y$ is $\varepsilon_{1}$-robust transversality if for all $\delta$-cubes $Q$, all vectors $V$,
    $$\#\{T\in\mathbb{T}_{Y}(Q):\angle(V(T),V)\leq r\}\lesssim r^{\varepsilon_{1}}\mu_{Y}, $$
    for any $\delta<r<1$.
\end{definition}

When performing incidence estimates, the situation can often be reduced to the two cases mentioned above; see \cite[Proposition 2.9 and Proposition 2.12]{katz21} for a detailed treatment. The following definition describes the number of tubes in the same direction. 

\begin{definition}[$m$-parallel]
We say a family of $\delta$-tubes $\mathbb{T}$ is $m$-parallel if there are $\lesssim m$ tubes pointing in the direction $V$ for every $\delta$-ball $V$ in $S^{n-1}$. 
\end{definition}

Now, we recall the incidence estimate established in \cite{WangWu2024} and \cite{GWX}, which was obtained by combining the two-ends Furstenberg estimate in the plane \cite{WangWu2024,wangwu2025} with the hairbrush argument \cite{Wolff95}.
\begin{lemma} \cite{WangWu2024,GWX}\label{lem hairbrush}
    Let $(\mathbb{T},Y)$ be a set of tubes and shading in $\mathbb{R}^{4}$, and let $m\geq 1$. Suppose that $\mathbb{T}$ is an $m$-parallel family, and $Y$ is an $(\varepsilon_{1},\varepsilon_{2})$-two-ends, $\lambda$-dense shading. Let $E_{Y}=\bigcup_{T\in\mathbb{T}}Y(T)$. For any $x\in E_Y$, define $$\mathbb{T}_{Y}(x)=\{T\in\mathbb{T}:x\in Y(T)\}.$$ Then for all $\varepsilon>0$, there exists a constant
    $c_{\varepsilon}$ such that
    $$|E_{Y}|\geq c_{\varepsilon}\delta^{\varepsilon}\delta^{\varepsilon_{1}}m^{-\frac{1}{2}}\lambda^{\frac{7}{4}}\delta^{1}(\delta^{3}\#\mathbb{T})^{\frac{1}{2}}.$$
 Therefore, the average multiplicity satisfies the following estimate
 \begin{equation}\label{eq: mulhairbrush}
     \mu_{Y}\lessapprox \delta^{-\varepsilon_{1}}m^{\frac12} \lambda^{{-\frac34}}\delta^{-1}(\delta^{3}\#\mathbb{T})^{\frac12}.
 \end{equation}
\end{lemma}

Next, we revisit the trilinear Kakeya estimate in $\mathbb{R}^{4}$, which was established by Guth and Zahl \cite{Guth-Zahl}.

\begin{theorem} \cite{Guth-Zahl}\label{thm GZ}
Let $\mathbb{T}$ be a set of $\delta$-tubes in $\mathbb{R}^{4}$ that point in $\delta$-separated directions. Then for any $\varepsilon>0$, there exists a constant $C_{\varepsilon}$ such that 
\begin{equation}
    \int (\sum_{T',T'',T'''\in \mathbb{T}}\chi_{T'}\chi_{T''}\chi_{T'''}|V(T')\wedge V(T'')\wedge V(T''')|^{\frac{12}{13}})^{\frac{13}{27}} \leq C_{\varepsilon} \delta^{-\frac{1}{3}-\varepsilon}(\delta^{3}\#\mathbb{T})^{\frac{4}{3}}, 
\end{equation}
where $V(T)$ denotes the direction of the tube $T$. 
\end{theorem}
Furthermore, the above theorem leads to a corollary concerning quantitatively transverse shading, which establishes an upper bound for the average multiplicity.

\begin{corollary}\label{cor GZ}
    Let $(\mathbb{T},Y)$ be a family of $1$-parallel tubes with $\lambda$-dense shading that satisfies 
    $$|\{T',T'',T'''\in \mathbb{T}_{Y}(Q):|V(T')\wedge V(T'')\wedge V(T''')|\gtrsim \rho\} |\gtrapprox |\mathbb{T}_{Y}(Q)|^{3}.$$
    Then the average multiplicity has the following estimate
    $$\mu_{Y}\lessapprox \lambda^{-\frac{9}{4}}\rho^{-1}\delta^{-\frac{3}{4}}(\delta^{3}\#\mathbb{T})^{\frac{3}{4}}.$$
\end{corollary}
\begin{proof}
We sketch the proof. For a $1$-parallel shading $(\mathbb{T}',Y)$, we have
\begin{equation}\label{eq trilinearincidence}
\left|\bigcup_{T\in \mathbb{T}'}Y(T)\right|\gtrapprox \lambda_{Y}^{\frac{13}{4}}\delta^{\frac{3}{4}}\rho(\delta^{3}\#\mathbb{T}')^{\frac{1}{4}},
\end{equation}
by \cite[Corollary 3.6]{katz21}.
By double counting the same incidence, we have the relation
\begin{equation}\label{eq doublecounting}
\Big|\bigcup_{T\in\mathbb{T}}Y(T)\Big|=\mu_Y^{-1}\lambda_{Y}\delta^{3}\#\mathbb{T}.
\end{equation}
Combining \eqref{eq trilinearincidence} with \eqref{eq doublecounting}, we obtain
\begin{equation}
\mu_{Y}\lessapprox \lambda_{Y}^{-\frac{9}{4}}\rho^{-1}\delta^{-\frac{3}{4}}(\delta^{3}\#\mathbb{T})^{\frac{3}{4}}.
\end{equation}
\end{proof}
For the plany case, we will use the following planebrush estimate, which was established by Katz and Zahl \cite{katz21}. It is worth noting that this estimate is obtained through a geometric argument, which often provides a more favorable power of $\lambda$ under the two-ends assumption.

\begin{lemma} \cite{katz21}\label{lem planebrush}
Let $0<\rho<1$. Let $\Omega\subset S^{3}$ be a set of $\rho$-separated directions. Let $(\mathbb{T},Y)$ be a family of essentially distinct plany $\rho$-tubes with with $\lambda$-dense shading.
Suppose that
\begin{itemize}
    \item The tube family $\mathbb{T}$ is $A$-parallel.
    \item The shading $(\mathbb{T},Y)$ satisfies $(\varepsilon_{1},\varepsilon_{2})$-two-ends condition.
    \item There are numbers $D,D_{1},D_{2}$ with $D_{1}\geq D_{2}$ and $\frac{D_{1}}{D_{2}}\leq D$ such that the following holds. For each $Q \in \bigcup_{T \in \mathbb{T}} Y(T)$, there exist planes $\Pi_1(Q), \dots, \Pi_{D_1}(Q)$ and collections of tubes $\mathbb{T}^1(Q), \dots, \mathbb{T}^{D_1}(Q) \subset \mathbb{T}_{Y}(Q)$ such that for each index $1 \leq i \leq D_1$, the tubes in $\mathbb{T}^i(Q)$ satisfy $\angle(T, \Pi_i(Q)) \leq \rho$ and $|\mathbb{T}^i(Q)| = \frac{\mu_Y}{D}$. Moreover, each tube $T \in \mathbb{T}_{Y}(Q)$ is contained in exactly $D_2$ of the collections $\{\mathbb{T}^i(Q)\}$.
\end{itemize}
Then we have the following estimate,
$$\Big|\bigcup_{T\in\mathbb{T}}Y(T)\Big|\gtrapprox \rho^{\varepsilon_{1}}\lambda^{1/3}\rho^{2/3}D^{-4/3}A^{-1/3}\sum_{T\in\mathbb{T}}|Y(T)|.$$
\end{lemma}

\subsection{Oscillatory estimates} 
We employ the method of Wang and Wu for handling the oscillatory part, making estimates at the $L^{2}$ and $L^{2\frac{n+1}{n-1}}$ endpoints and then applying interpolation.

\begin{lemma}\cite{WangWu2024}
    Let $X=\cup_{Q\in\mathcal{Q}}Q$ be the union of a collection $\mathcal Q$ of $R^{\frac{1}{2}}$-cubes, and let $f=\sum_{T\in\mathbb{T}}f_{T}$ be a sum of wave packets. Suppose for each $T\in \mathbb{T}$, there is a shading $Y(T)\subset T$ by
    $R^{\frac{1}{2}}$-cubes in $\mathcal{Q}$ such that the number of $R^{\frac{1}{2}}$-cubes contained in $Y(T)$ is at most $\lambda R^{\frac{1}{2}}$. Then 
    \begin{equation}
        \int_{X}\Big|\sum_{T\in\mathbb{T}}\mathcal Ef_{T}\mathbf{1}_{Y(T)}\Big|^{2}\lesssim\lambda R\|f\|_{L^{2}}^{2}.
    \end{equation}
\end{lemma}

By the Young inequality, we have the following lemma.
\begin{lemma}\label{lem TomasStein}
    In $\mathbb{R}^{n}$, we have 
    $$\|E f_{T}\|_{L^{2\frac{n+1}{n-1}}(\omega_{B_{R}})}\lesssim R^{-\frac12}  \|\mathcal Ef_{T}\|_{L^2(\omega_{B_{R}})},$$
where $T$ is a wave packet and $\omega_{B_{R}}$ is a weight that is $\approx 1$ on $B_{R}$ and decays rapidly outside the ball $B_{R}$. 
\end{lemma}

Finally, we recall the refined decoupling inequality. It was first proved in \cite{GIOW} for the paraboloid in the translation-invariant setting and was independently observed by Du and Zhang \cite{DuZhang}. In the Wang–Wu method, this inequality serves as a key tool for handling the oscillatory part of restriction estimates, effectively converting the restriction problem into an incidence estimate.
\begin{theorem}[Refined decoupling]\label{thm refineddecoupling}
    Let $\mathcal E_{S}$ be the Fourier extension operator, where the hypersurface $S\subset \mathbb{R}^{n}$ has a strictly positive second fundamental form, and let $p_{n}=2\frac{n+1}{n-1}$. Suppose $f$ is a sum of wave packets $f=\sum_{T\in \mathbb{T}} f_{T}$ such that $\|\mathcal Ef_{T}\|_{L^{p_{n}}(\omega_{B_{R}})}^{2}$ are the same up to a constant multiple for all $T\in \mathbb{T}$. Let $X$ be a union of $R^{1/2}$-balls in $B_{R}$ such that each $R^{1/2} $-ball $Q\subset X$ intersects at most $M$ tubes $T$. Then

$$\|\mathcal E_{S}f\|_{L^{p_{n}}(X)}^{p_{n}}\lessapprox M^{\frac{2}{n-1}}\sum_{T\in\mathbb{T}}\|\mathcal E_{S}f_{T}\|_{L^{p_n}(\omega_{B_{R}})}^{p_n}.$$
 Here $\omega_{B_{R}}$ is a weight that is $\approx 1$ on $B_{R}$ and decay rapidly outside the ball $B_{R}$.   
\end{theorem}

\begin{remark}\label{rem BR} 
By classical reductions (see, e.g., \cite{CS,GHI}), the Bochner--Riesz statement in Corollary \ref{cor BochnerRiesz} reduces to bounds for the Carleson--Sjölin operators 
\begin{equation}\label{eq CSoperator}
    S_{N}f(x)\coloneqq \int_{\mathbb{R}^{n}} e^{2\pi i N|x-y|}\, a(x,y)\, f(y)\, dy,
\end{equation}
where $a\in C^{\infty}(\mathbb{R}^{n}\times\mathbb{R}^{n})$ has compact support away from the diagonal. By rescaling, set $S^{N}f(x)\coloneqq S_{N}f(x/N)$.

It is well known that refined decoupling holds for positive-definite H\"ormander phases \cite{ILX}, of which the Euclidean distance $|x-y|$ is a special case. Therefore, the refined decoupling estimates in Theorem \ref{thm refineddecoupling} apply to \eqref{eq CSoperator}. Moreover, \cite{GWX} shows that the Wang--Wu decoupling--incidence method extends, via a standard $n\times n$ wave-packet decomposition, to positive-definite H\"ormander operators. For the Euclidean distance phase $|x-y|$, the wave packet family of $S^{N}$ consists of straight tubes and satisfies the same incidence estimate as Lemma \ref{lem incidence}.

Analogous to the proof of Theorem \ref{thm: restriction}, the decoupling--incidence argument in this paper then yields the operator bound
\begin{equation}\label{eq BR}
\|S^{N}f\|_{L^{p}(\mathbb{R}^{n})} \le C_{\varepsilon} N^{\varepsilon} \|f\|_{L^{p}(\mathbb{R}^{n})}
\end{equation}
for any $p\ge 2+\frac{200}{251}$ and any $\varepsilon>0$, with $C_{\varepsilon}$ independent of $N$. This implies Corollary \ref{cor BochnerRiesz}. This strategy was outlined in \cite{GWX} to improve the Bochner--Riesz exponent to $22/7$ in three dimensions.
\end{remark}

\section{The proof of Theorem \ref{thm: restriction}}

By Tao's $\varepsilon$-removal argument \cite{Tao99}, to obtain Theorem \ref{thm: restriction}, it suffices to prove the following localized restriction theorem.
\begin{theorem}\label{restriction}
    For any $\varepsilon>0$, there exists a constant $C_{\varepsilon}$ such that
    $$\|\mathcal Ef\|_{L^{p}(B_{R})}\leq C_{\varepsilon}R^{\varepsilon}\|f\|_{L^p(B_1^{3})},$$
    for any $p\geq2+\frac{200}{251}$.
\end{theorem}

We consider the standard wave-packet decomposition. Let $\Theta$ be a covering of the unit ball $B_{1}^{3}$ by $R^{-1/2}$-balls with finite overlap. We define a smooth partition of unity $\{\eta_{\theta}\}_{\theta \in \Theta}$ subordinate to this covering, such that $\operatorname{supp}(\eta_{\theta}) \subset 2\theta$ and $\sum_{\theta \in \Theta} \eta_{\theta} = 1$ on $B_{1}^{3}$.

Similarly, in physical space, let $\Xi$ be a finitely overlapping covering of $\mathbb{R}^{3}$ by $R^{-1/2}$-balls, and let $\{\tau_{v}\}_{v \in \Xi}$ be a smooth partition of unity in $\mathbb{R}^{3}$ such that $\operatorname{supp}(\widehat{\tau}_{v}) \subset B^{3}(0, R^{-1/2})$ and $\sum_{v \in \Xi} \tau_{v} = 1$ in $\mathbb{R}^{3}$. 

Define  $f_{\theta,v}\coloneqq(f\eta_{\theta})*\hat\tau_{v}$. Then we perform the following wave packet decomposition
\begin{equation}\label{eq Wavepacket}
    f=\sum_{\theta\in\Theta}\sum_{v\in\Xi}f_{\theta,v}.
\end{equation}
For each $\theta \in \Theta$ and $v \in \Xi$, we define a tube 
$$
T_{\theta,v} = \Big\{ (x, t) \in B^{4}(0, R) : |x - c_{v} + t \nabla \psi(c_{\theta})| \leq R^{\frac{1}{2} + \varepsilon_{0}} \Big\}.
$$
Each tube $T_{\theta,v}$ has dimensions $R^{1/2+\varepsilon_{0}} \times R^{1/2+\varepsilon_{0}} \times R^{1/2+\varepsilon_{0}} \times R$. Here, $c_{\theta}$ and $c_{v}$ denote the centers of $\theta$ and $v$ respectively, and $\varepsilon_{0} > 0$ is a small constant. Let $V(\theta)$ be the vector $(1,\nabla \psi(c_{\theta}))$. Let ${\mathbb{T}}(\theta)=\{T_{\theta,v}:v\in\Xi \quad \text{and}\quad T_{\theta,v}\cap B_{R}\neq \emptyset \}$ be a family of $R$-tubes with direction $V(\theta)$. We denote $f_{\theta,v}$ by $f_{T}$ if $T=T_{\theta,v}$ and denote the tube family by $\mathbb{T}$. Therefore, we have
\begin{equation}
    \|\mathcal Ef\|_{L^{p}(B_{R})}\lesssim \Big\|\sum_{T\in\mathbb{T}}\mathcal Ef_{T}\Big\|_{L^{p}(B_{R})}.
\end{equation}

Let $f_{\theta}\coloneqq f\eta_{\theta}$ and define $$\|f_{\theta}\|_{L^{2}_{avg(\theta)}}^{2}\coloneqq \frac{1}{|\theta|}\|f_{\theta}\|_{2}^{2}.$$ In particular, we prove a stronger version of Theorem \ref{restriction}. We approximate the $L^p$-norm by $\|f\|_{2}^{2/p} \sup_{\theta \in \Theta} \|f_{\theta}\|_{L^{2}_{\text{avg}}(\theta)}^{1-2/p}$, following the approach introduced by Guth \cite{guth2016restriction}.

\begin{theorem}\label{thm st}
    For any $\varepsilon>0$, there exists a constant $C_{\varepsilon}$ such that
    $$\|\mathcal Ef\|_{L^{p}(B_{R})}^{p}\leq C_{\varepsilon}R^{\varepsilon}\|f\|_{2}^{2}\sup_{\theta\in \Theta}\|f_{\theta}\|_{L^{2}_{avg}(\theta)}^{p-2}$$
    for $p=2+\frac{200}{251}$.
\end{theorem}

We prove the theorem by induction on scales. For $R\geq 1$, define $S(R)$ to be the smallest constant such that
\[
\|\mathcal Ef\|_{L^{p}(B_{R})}^{p} \leq S(R) \|f\|_{2}^{2} \sup_{\theta \in \Theta} \|f_{\theta}\|_{L^{2}_{\text{avg}}(\theta)}^{p-2}.
\]
Our goal is to show that for every $\varepsilon > 0$ and every $R\geq 1$, there exists a constant $C_{\varepsilon}$ such that
\begin{equation}\label{SR}
    S(R) \leq C_{\varepsilon} R^{\varepsilon}.
\end{equation}
By choosing $C_{\varepsilon}$ sufficiently large, inequality \eqref{SR} holds for all $1 \leq R \leq 100$.

We now proceed by induction on $R$. Assume that for all $\bar{R}$ satisfying $1 \leq \bar{R} \leq \frac{R}{2}$, there exists a constant $C_{\varepsilon}$ such that
\begin{equation}\label{eq induction hypothesis}
 S( \bar{R}) \leq C_{\varepsilon} \bar{R}^{\varepsilon}.   
\end{equation}
We will show that the same bound holds at scale $R$.

\subsection{Incidence estimate}

The oscillation is treated via refined decoupling; the remaining task is to establish a better lower bound for the multiplicity. For this, Wang and Wu \cite{WangWu2024} established an incidence estimate by combining the two-ends Furstenberg inequality in the plane and hairbrush argument. Their work, along with \cite{GWX}, provides an incidence estimate in $\mathbb{R}^4$ at dimension $3$ that possesses a good power of $\lambda$, which implies the restriction estimate for the range of $p \ge 14/5$ matching a well-known threshold in the classical restriction estimate \cite{Guth18} and positive-definite Hörmander operators \cite{GHI}.

Guth and Zahl \cite{Guth-Zahl} and Katz and Zahl \cite{katz21} surpassed the barrier of dimension $3$ for the Kakeya maximal estimate by employing the plany/transverse argument. Their work inspires us to break the threshold $p \geq 14/5$ in the restriction problem. However, the incidence estimates in \cite{Guth-Zahl, katz21} lack a favorable power of $\lambda$, which prevents them from being directly applied to the restriction estimate.

Fortunately, in the plany scenario, the planebrush argument developed in \cite{katz21} yields an incidence estimate with a good power of $\lambda$. This allows us to trade a higher power of $\delta$ for a smaller power of $\lambda$.

We establish an improved two-ends incidence estimate in $\mathbb{R}^4$ with a better exponent in $\delta$ compared to the hairbrush estimate and a more favorable exponent in $\lambda$, thereby fitting into the restriction framework of \cite{WangWu2024}.

\begin{lemma}\label{lem incidence}
    In $\mathbb{R}^{4}$, for any $m$-parallel shading $(\mathbb{T},Y)$ satisfying the $(\varepsilon_{1},\varepsilon_{2})$-two-ends condition and $\lambda$-dense assumption, we have
    \begin{equation}\label{eq incidence}
        |\bigcup_{T\in\mathbb{T}}Y(T)|\gtrapprox \delta^{\varepsilon_{1}}m^{-\frac{9}{10}}\lambda^{\frac{201}{100}}\delta^{\frac{49}{50}}(\delta^{3}\#\mathbb{T})^{\frac{9}{10}}. 
    \end{equation}
\end{lemma}

We recall the following technical lemma that establishes the connection between the plany case and the quantitatively transverse case, and we will use it in the proof of Lemma \ref{lem incidence}.

\begin{lemma} \cite[Lemma 6.1]{katz21}\label{lem dich}
    Let $(\mathbb{T},Y)$ be a set of $\delta$-tubes and their associated shading. Suppose that $(\mathbb{T},Y)$ is $(\varepsilon_{1},\varepsilon_{2})$-two-ends and $\varepsilon_{3}$-robustly transverse.
    
Then there exists a number $\delta\leq \rho\leq 1$; a refinement $(\mathbb{T}_{1},Y_{1})$ of $(\mathbb{T},Y)$; a set $\Omega\subset S^{3}$ of $\theta$-separated points; a set $(\mathbb{T}_{\rho},Y_{\rho})$ of essentially distinct $\rho$-tubes and their associated shading; and numbers $D,D_{1},D_{2}\geq1$ with $D_{1}\geq D_{2}$ and $\frac{D_{1}}{D_{2}}\leq D$ 
so that the following holds
\begin{enumerate}[label=(\roman*)]

    \item \label{item dich i}
    For each cube $Q\in \bigcup_{T\in \mathbb{T}_{1}}Y_{1}(T)$, there is a plane $\Pi(Q)$ so that $\angle(V(T),\Pi(Q))\leq \rho$ for all $T\in \mathbb{T}_{Y_1}(Q)$, where $V(T)$ denotes the direction of the tube $T$.
    
    \item \label{item dich ii}
    For each $Q\in \bigcup_{T\in \mathbb{T}_{1}}Y_{1}(T)$, 
     $$|\{T',T'',T'''\in \mathbb{T}_{Y_{1}}(Q):|V(T')\wedge V(T'')\wedge V(T''')|\gtrsim \rho\} |\gtrapprox |\mathbb{T}_{Y_{1}}(Q)|^{3}.$$

     \item \label{item dich iii}
     Every tube in $\mathbb{T}_{\rho}$ points in a direction from $\Omega$. For each $\omega \in \Omega$, there are $\approx A$ tubes from $\mathbb{T}_{\rho}$ in direction $\omega$.

     \item \label{item dich iv}
    Each tube $T_{\rho}\in\mathbb{T}_{\rho}$ contains $\approx \frac{|\mathbb{T}_{1}|}{|\mathbb{T_{\rho}}|}$ tubes from $\mathbb{T}_{1}$, and each tube from $\mathbb{T}_{1}$ is contained 
     in exactly one tube from $\mathbb{T}_{\rho}$.

    \item \label{item dich v}
    If $T\in\mathbb{T}_1$, $T_{\rho}\in \mathbb{T}_{\rho}$, and $T\subset T_{\rho}$, then $Y_{1}(T)\subset Y_{\rho}(T_{\rho})$.

    \item \label{item dich vi}
    $$\lambda_{Y_{\rho}}\gtrapprox \lambda_{Y_{1}}\gtrapprox \lambda_{Y}.$$

    \item \label{item dich vii}
    $(\mathbb{T}_{\rho},Y_{\rho})$ is $(\varepsilon_{1},\varepsilon_{2})$-two-ends.

    \item \label{item dich viii}
    For each $\delta$-cube $Q\in \bigcup_{T\in\mathbb{T}_1}Y_{1}(T)$, we have $|\mathbb{T}_{Y_{1}}(Q)|\approx \mu_{Y_{1}}$. For each $\rho$-cube $Q_{\rho}\in \bigcup_{T_{\rho}\in\mathbb{T}_{\rho}}Y_{\rho}(T_{\rho})$, we have $|\mathbb{T}_{\rho}(Q_{\rho})|\approx \mu_{Y_{\rho}}$.

    \item \label{item dich ix}
    There is a number $\tilde \mu$ so that for each $T_{\rho}\in\mathbb{T}_{\rho}$, we have $\mu_{Y_{T_{\rho}}}\approx \tilde \mu $, where $\mu_{Y_{T_{\rho}}}$ is the multiplicity of 
    $\delta$-tubes in a thicker tube $T_{\rho}$. 
    Moreover, $\mu_{Y_1}\leq\rho^{-1}\tilde{\mu}$. 

    \item \label{item dich x}
    For each $Q_{\rho}\in \bigcup_{T_{\rho}\in\mathbb{T}_{\rho}}Y_{\rho}(T_{\rho})$, there are planes $\Pi_{1}(Q_{\rho}),\cdots,\Pi_{D_{1}}(Q_{\rho})$ and collections of $\rho$-tubes
    $\mathbb{T}_{\rho}^{1}(Q_{\rho}),\cdots,\mathbb{T}_{\rho}^{D_{1}}(Q_{\rho})\subset \mathbb{T}_{\rho}(Q_{\rho})$ so that the $\rho$-tubes in $\mathbb{T}_{\rho}^{i}(Q_{\rho})$ satisfy
    $\angle(T_{\rho},\Pi_{i}(Q_{\rho}))\leq \rho$ and $|\mathbb{T}_{\rho}^{i}|=\frac{\mu_{Y_{\rho}}}{D}$ for any index $1\leq i\leq D_{1}$. Each tube $T\in \mathbb{T}_{\rho}(Q_\rho)$ is contained in $D_{2}$ collections of the form ${\mathbb{T}_{\rho}^{i}(Q_{\rho})}$.

    \item \label{item dich xi}
    We have the multiplicity bound

        $$\mu_{Y_{1}}\lessapprox \frac{\mu_{Y_{\rho}} \cdot \tilde \mu}{D}.$$
        
\end{enumerate}

\end{lemma}

For the proof of Lemma \ref{lem dich}, we refer the reader to \cite[Lemma 6.1]{katz21}. 
Intuitively, this lemma first finds a refinement $(\T_1, Y_1)$ of the shaded $\delta$-tubes $(\T, Y)$ and a larger scale $\rho\in[\delta, 1]$ such that 
for each shaded $\delta$-cube 
$Q\in\bigcup_{T\in\T_1}Y_1(T)$, 
the $\delta$-bush centered at $Q$
is ``$\rho$-plany'' (property \ref{item dich i}), i.e., 
it clusters inside a $\rho\times1\times1\times1$ plank centered at $Q$,
and inside this plank the bush is ``trilinearly transverse'' (property \ref{item dich ii}).  
The multiplicities of the bushes are dyadically pigeonholed to be uniform (property \ref{item dich viii}).  

Then the lemma indicates that for this larger scale $\rho\in[\delta, 1]$, the collection $\T_1$ is packed evenly and efficiently (property \ref{item dich iv}) into an $A$-parallel family of $\rho$-tubes $\T_\rho$ (property \ref{item dich iii}). The shading $Y(\T_\rho)$ is inherited (property \ref{item dich v}) from $Y(\T)$ and preserves the density (property \ref{item dich vi}) and the two-ends property (property \ref{item dich vii}). 
Across all the shaded $\rho$-tubes $Q_\rho\in Y(\T_\rho)$, the multiplicity of the $\rho$-bush is uniform (property \ref{item dich viii}). 
The multiplicity of the bushes of $\delta$-tubes inside each thicker $\rho$-tube is also dyadically pigeonholed to be uniform inside and across every $\rho$-tube (property \ref{item dich ix}). 
Moreover, since each $\delta$-bush is ``$\rho$-plany'', it is covered by at most $\rho^{-1}$ many $\rho$-tubes (property \ref{item dich ix}). 

For each $Q_\rho\in Y(\T_\rho)$, 
the $\rho$-bush
centered at $Q_\rho$ 
is packed into $D_1$ many roughly $\rho\times1\times1\times1$ planks, 
in which each $\rho$-tube is packed $D_2$ times by such planks
(property \ref{item dich x}). 
The multiplicity $\mu_{Y_1}$ is 
simply\footnote{
    Taking into account the graininess inside the shaded $\rho$-tubes could potentially provide sharper control on  $\mu_{Y_1}$ than the bound in property \ref{item dich xi}. 
}
controlled (property \ref{item dich xi}) by first counting at the larger scale $\rho$ and then counting at scale $\delta$ inside each packed $\rho$-tube, modulo the overcounted quantity $D\geq\frac{D_1}{D_2}$.

\begin{proof}[Proof of Lemma \ref{lem incidence}]
  We first prove the lemma in the $1$-parallel scenario. Let $(\tilde {\mathbb{T}},Y)$ be a 
1-parallel family of tubes with shading. It suffices to prove the lemma when $(\mathbb{\tilde T}, Y)$ is $\varepsilon_3$-robustly transverse, and there is a standard argument to remove the robust transverse assumption. (See \cite[Proposition 2.12]{katz21}.) 
Let 
the refined family of $\lambda_{Y_{1}}$-densely shaded tubes $(\T_1, Y_1)$,
the scale 
$\rho\in[\delta,1]$, 
the set $\Omega\in S^3$ of $\rho$-separated directions, 
the family $\T_\rho$ of essentially distinct $\rho$-tubes, 
the shading $Y_\rho$ on $\T_\rho$, 
the uniform multiplicity $\tilde\mu$ of $\delta$-tubes inside each $T_\rho\in\T_\rho$,
and the positive parameter $D$ be the outputs of 
{Lemma \ref{lem dich}}.

We apply the lemma \ref{lem hairbrush} to the $A$-parallel shading $(\mathbb{T}_{\rho}, Y_{\rho})$ to obtain the upper bound for the average multiplicity
\begin{equation}\label{eq rho}
    \mu_{Y_{\rho}}\lessapprox \rho^{-\varepsilon_{1}}A^{\frac12}\lambda_{Y_{\rho}}^{-\frac{3}{4}}\rho^{-1}(\rho^{3}\#\mathbb{T}_{\rho})^{\frac12}.
\end{equation}

For the plany shading $(\mathbb{T}_{\rho},Y_{\rho})$, we use the incidence estimate obtained by the planebrush argument (See Lemma \ref{lem planebrush})
\begin{equation}\label{eq planebrush}
    \mu_{Y_{\rho}} \lessapprox \rho^{-\varepsilon_{1}} 
    \lambda_{Y_{\rho}}^{-\frac{1}{3}} 
    \rho^{-\frac{2}{3}} 
    A^{\frac{1}{3}} 
    D^{\frac{4}{3}}.
\end{equation}

Combining \eqref{eq planebrush} and \eqref{eq rho} with weights $3/4$ and $1/4$, respectively, we have
\begin{equation} \label{eq a}
    \mu_{Y_{\rho}} \lessapprox \rho^{-\varepsilon_{1}}\lambda_{Y_{\rho}}^{-\frac{7}{16}} \rho^{-\frac{3}{4}} A^{\frac{3}{8}} (\rho^{3} \# \mathbb{T}_{\rho})^{\frac{1}{8}}D.
\end{equation}

Furthermore, using Lemma \ref{lem dich} (xi) and the above inequality, we have the following
\begin{equation}\label{eq relation}
     \mu_{Y_{1}}\lessapprox \rho^{-\varepsilon_{1}} \lambda_{Y_{\rho}}^{-\frac{7}{16}}\rho^{-\frac{3}{4}} A^{\frac{3}{8}}(\rho^{3}\#\mathbb{T}_{\rho})^{\frac{1}{8}}\tilde{\mu}.
\end{equation}

Now we compute the number $\tilde \mu$. For each $T_{\rho}\in \mathbb{T}_{\rho}$, we rescale the $\rho$-tube to a unit cube. Naturally, the $\delta$-tubes contained in $T_{\rho}$ become the $\frac{\delta}{\rho}$-tubes. Define the tube collection
$$\mathbb{T}_{\delta/\rho}\coloneqq\{T\in \mathbb{T}_{1}:T \subset T_{\rho} \},$$ 
for each $\rho$-tube $T_{\rho}$, and we note $\#\mathbb{T}_{\delta/\rho}\approx \frac{\#\mathbb{T}_{1}}{\#\mathbb{T}_{\rho}}$. 

After rescaling, $(\mathbb{T}_{\delta/\rho},Y_{1})$ induces a new shading 
$(\bar {\mathbb{T}}_{\delta/\rho},\bar Y_{1})$. Since $(\mathbb{T},Y)$ is $1$-parallel shading, the shading $(\bar {\mathbb{T}}_{\delta/\rho},\bar Y_{1})$ is $1$-parallel and $\lambda_{Y_{1}}$-dense. Additionally, by Lemma \ref{lem dich} vi, we have $\lambda_{Y_{\rho}}^{-1} \lessapprox \lambda_{Y_{1}}^{-1} \lessapprox \lambda^{-1}$. We therefore replace the dense parameters $\lambda_{Y_{1}}$ and $\lambda_{Y_{\rho}}$ with $\lambda$ in the following.

Apply Lemma \ref{lem hairbrush} for the shading $(\bar {\mathbb{T}}_{\delta/\rho},\bar Y_{1})$, then we obtain the upper bound of average multiplicity
\begin{equation}\label{eq tilde}
    \tilde{\mu} \lessapprox 
    \left(\frac{\delta}{\rho}\right)^{-\varepsilon_{1}}
    \lambda^{-\frac{3}{4}}
    \left(\frac{\delta}{\rho}\right)^{-1}
    \left(
\frac{\delta^3\#\mathbb{T}_{1}}{\rho^3\#\mathbb{T}_{\rho}}
    \right)^{\frac{1}{2}}.
\end{equation}

Recalling that the $\delta$-tube family $\T_{1}$ is $1$-parallel, we have the following relation
\begin{equation}\label{eq rho-1}
    \frac{\delta^{3}\#\T_{1}}{\rho^{3}\#\T_{\rho}}\lesssim \frac{1}{A}.
\end{equation}

Combining this with the inequalities \eqref{eq relation} and \eqref{eq tilde}, and simplifying, we obtain
\begin{equation}\label{eq 1}
    \mu_{Y_{1}} \lessapprox 
    \delta^{-\varepsilon_{1}}
    \lambda^{-\frac{7}{16}}
    \rho^{\frac{1}{4}}
    \lambda^{-\frac{3}{4}}
    \delta^{-1}
\bigl(\delta^{3}\#\mathbb{T}_{1}\bigr)^{\frac{1}{8}}.
\end{equation}

Since the $\delta$-tubes are contained in a $\rho$-neighborhood of a plane, we obtain
\begin{equation}\label{eq 2}
    \mu_{Y_{1}}\lesssim \rho^{-1}\tilde \mu \lessapprox \delta^{-\varepsilon_{1}} \lambda^{-\frac{3}{4}}\delta^{-1}\left(\frac{\delta^{3}\# \mathbb{T}_{1}}{\rho^{3}\#\mathbb{T}_{\rho}}\right)^{\frac{1}{2}}\lessapprox \delta^{-\varepsilon_{1}}\lambda^{-\frac{3}{4}}\delta^{-1}\left(\frac{1}{A}\right)^{\frac{1}{2}}.
\end{equation}
Taking the weighted combination \eqref{eq 1}$^{\frac{8}{23}}\cdot$\eqref{eq 2}$^{\frac{15}{23}}$, we derive
\begin{equation}\label{eq theta}
    \mu_{Y_{1}}\lessapprox \delta^{-\varepsilon_{1}}\lambda^{-\frac{7}{46}}\rho^{\frac{2}{23}} \lambda^{-\frac{3}{4}}\delta^{-1}(\delta^{3}\#\mathbb{T}_{1})^{\frac{1}{23}}.
\end{equation}
We now remove the parameter $\rho$ by combining our estimate with the Guth--Zahl $3$-linear estimate (Theorem \ref{thm GZ} and Corollary \ref{cor GZ})
\begin{equation}\label{eq GZ}
    \mu_{Y_{1}}\lessapprox \delta^{-\varepsilon_{1}}  \lambda^{-\frac{9}{4}}\rho^{-1}\delta^{-\frac{3}{4}}(\delta^{3}\#\mathbb{T}_{1})^{\frac{3}{4}}.
\end{equation}
Taking the weighted combination \eqref{eq theta}$^{\frac{23}{25}}\cdot$ \eqref{eq GZ}$^{\frac{2}{25}}$, we obtain
\begin{equation}
    \mu_{Y_{1}}\lessapprox \delta^{-\varepsilon_{1}}\lambda^{-\frac{101}{100}}\delta^{-\frac{49}{50}}(\delta^{3}\#\mathbb{T}_{1})^{\frac{1}{10}}.
\end{equation}

Since the shading $(\mathbb{T}_{1},Y_{1})$ is a refinement of $(\mathbb{T},Y)$, then we obtain the incidence estimate
$$|\bigcup_{T\in\tilde {\mathbb{T}}}Y(T)|\geq |\bigcup_{T\in\mathbb{T}_{1}}Y_{1}(T)|=\frac{1}{\mu_{Y_{1}}}\sum_{T\in\mathbb{T}_{1}}|Y_{1}(T)|\gtrapprox \delta^{\varepsilon_{1}}\lambda^{\frac{201}{100}}\delta^{\frac{49}{50}}(\delta^{3}\#\mathbb{T})^{\frac{9}{10}},$$
with the last inequality follows from Lemma \ref{lem dich} vi.

For the $m$-parallel shading $(\mathbb{T},Y)$, we have the following relation
$$\#\bar{\mathbb{T}}\gtrapprox\frac{\#\mathbb{T}}{m}.$$
Therefore, we obtain
$$|\bigcup_{T\in\mathbb{T}}Y(T)|\geq |\bigcup_{T\in\tilde {\mathbb{T}}}Y(T)|\gtrapprox \delta^{\varepsilon_{1}}\lambda^{\frac{201}{100}}\delta^{\frac{49}{50}}(\delta^{3}\#\bar{\mathbb{T}})^{\frac{9}{10}}\gtrapprox\delta^{\varepsilon_{1}}m^{-\frac{9}{10}}\lambda^{\frac{201}{100}}\delta^{\frac{49}{50}}(\delta^{3}\#\mathbb{T})^{\frac{9}{10}}.$$
We complete the proof of the lemma.
    
\end{proof}

In the proof of the lemma above, we estimate an upper bound for the average multiplicity. However, this average is not directly applicable to the restriction estimate. Instead, we will use the fact that the multiplicity is pointwise bounded, except on a set with small measure. We obtain the following property through the same proof as in \cite[Proposition 3.2]{WangWu2024}.

\begin{prop}\label{prop incidence}
      Let $\delta\in(0,1)$. Let $(\mathbb{T},Y)$ be an $m$-parallel, $\delta$-separated tube family with an $(\varepsilon_{1},\varepsilon_{2})$-two-ends, $\lambda$-dense shading. Let $E_{Y}=\bigcup_{T\in\mathbb{T}}Y(T)$. In particular, by taking $$\mu=\delta^{-3\varepsilon_{1}}m^{\frac{9}{10}}\lambda^{-\frac{101}{100}}\delta^{-\frac{49}{50}}(\delta^{3}\#\mathbb{T})^{\frac{1}{10}},$$ there exists a set $E_{\mu}\subset E_{Y}$ such that
    $\#\mathbb{T}_{Y}(x)\lessapprox\mu$ for all $x\in E_{\mu}$, and 
    \begin{equation}
        |E_{Y}\setminus E_{\mu}|\leq \delta^{\varepsilon_{1}}|E_Y|.
    \end{equation}
\end{prop}

With the improved incidence estimate established above in Lemma \ref{lem incidence}, the claimed restriction estimate follows directly from the Wang--Wu decoupling--incidence method. For completeness, however, we also provide a self-contained proof in the remainder of this section, using the exponents stated in Theorem \ref{thm: restriction}.

\subsection{Dyadic pigeonholing}
For two dyadic numbers $\beta_{1}$ and $\beta_{2}$, let 
$$\mathbb{T}_{\beta_{1},\beta_{2}}=\{T\in \mathbb{T}:\|f_{T}\|_2\approx \beta_1\|f\|_{2} \quad\text{and}\quad \|\mathcal Ef_{T}\|_{L^{\frac{10}{3}}(\omega_{B_{R}})} \approx \beta_{2}\|f\|_{2}\}.$$
By a standard pigeonholing argument in \cite[Section 5]{WangWu2024}, there exists a pair $(\beta_{1},\beta_{2})$ for which 
$$\Big\|\sum_{T\in\mathbb{T}}\mathcal Ef_{T}\Big\|_{L^{p}(B_{R})}\lessapprox \Big\|\sum_{T\in \mathbb{T}_{\beta_{1},\beta_{2}}}\mathcal Ef_{T}\Big\|_{L^{p}(B_{R})}.$$
For each $\theta\in \Theta$, let $\mathbb{T}_{\beta_{1},\beta_{2}}(\theta)=\mathbb{T}_{\beta_{1},\beta_{2}} \cap \mathbb{T}(\theta)$. There exist a
dyadic $m\geq 1$ and a subset $\Theta_{m}\subset \Theta$ such that $\#\mathbb{T}_{\beta_{1},\beta_{2}}(\theta)\approx m$ for all $\theta \in \Theta_{m}$. Let $\tilde {\mathbb{T}}=\cup_{\theta \in \Theta_{m}}\mathbb{T}_{\beta_{1},\beta_{2}}(\theta)$ and we have
$$\|\mathcal Ef\|_{L^{p}(B_{R})}\lessapprox\Big\|\sum_{T\in \tilde{\mathbb{T}}}\mathcal Ef_{T}\Big\|_{L^{p}(B_{R})}.$$

For simplicity, we relabel $\mathbb{T}_{\beta_1,\beta_2}(\theta)$ by $\mathbb{T}(\theta)$ and $\Theta_{m}$ by $\Theta$. Furthermore, we write $\bar{\mathbb{T}}$ as $\mathbb{T}$ for convenience. The tube family $\mathbb{T}$ satisfies the following properties:
\begin{itemize}
    \item $\#\mathbb{T}(\theta)\approx m$ for all $\theta \in \Theta$.
    \item $\|\mathcal Ef_{T}\|_{L^{\frac{10}{3}}(\omega_{B_{R}})}$ are comparable for all $T\in \mathbb{T}$.
    \item $\|f_{T}\|_{2}$ are comparable for all $T\in \mathbb{T}$.
\end{itemize} 

By dyadic pigeonholing, there is a family of disjoint $R^{1/2}$-balls $\{Q\}$ such that the values of $\|\sum_{T\in\mathbb{T}}\mathcal Ef_{T}\|_{L^{p}(Q)}$ are the same up to a constant multiple. Let $X=\bigcup Q$. Then we have  
\begin{equation}\label{eq reduction1}
\Big\|\sum_{T\in\mathbb{T}}\mathcal Ef_{T}\Big\|_{L^{p}(B_{R})}\lessapprox \quad \Big\|\sum_{T\in\mathbb{T}}\mathcal Ef_{T}\Big\|_{L^{p}(X)}.
\end{equation}

\subsection{Two-ends reduction}
For each $T\in \mathbb{T}$, we partition $T$ into sub-tubes $\mathcal{J}=\{J\}$ of length $R^{1-\varepsilon^{2}}$. Then we decompose the set $\mathcal{J}(T)$ as a disjoint union $\mathcal{J}(T)=\bigcup_{\lambda} \mathcal{J}_{\lambda}(T)$, where $\lambda\leq 1$ is a dyadic number and $|J\cap X|\approx \lambda R$ for any $J\in \mathcal{J}_{\lambda}(T)$.
Therefore, we have

\begin{equation}
    \sum_{T\in\mathbb{T}}\mathcal Ef_{T}=\sum_{\lambda}\sum_{T\in\mathbb{T}}\sum_{J\in \mathcal{J}_{\lambda}(T)} \mathcal Ef_{T}\mathbf{1}_{J}.
\end{equation}
For each $Q\subset X$, by pigeonholing, there is a dyadic number $\lambda(Q)$ such that
\begin{equation}
\Big\|\sum_{T\in\mathbb{T}}\mathcal Ef_{T}\Big\|_{L^{p}(Q)}^{p}\lessapprox \int_{Q}\Big|\sum_{T\in\mathbb{T}}\sum_{J\in\mathcal{J}_{\lambda(Q)}(T)}\mathcal Ef_{T}\mathbf{1}_{J}\Big|^{p}.
\end{equation}
Note that $\|\mathcal Ef_{T}\|_{L^{p}(Q)}$ are the same up to a constant multiple for all $Q \subset X$. By applying dyadic pigeonholing over the collection
\[
\Big\{ \lambda(Q),\ \Big\|\sum_{T\in\mathbb{T}}\sum_{J\in \mathcal{J}_{\lambda(Q)}(T)} \mathcal Ef_{T}\mathbf{1}_{J} \Big\|_{L^{p}(Q)} : Q\subset X \Big\},
\]
there exists a uniform dyadic number $ \lambda $ and a refinement $ X_1 \subset X $ such that $ \lambda(Q) = \lambda $ for all $ Q \subset X_1 $, and
\[
\Big\|\sum_{T\in\mathbb{T}}\sum_{J\in \mathcal{J}_{\lambda}(T)} \mathcal Ef_{T}\mathbf{1}_{J} \Big\|_{L^{p}(Q)}
\]
is comparable up to a constant multiple for all $ Q \subset X_1 $. Thus, 
$$\Big\|\sum_{T\in\mathbb{T}}\mathcal Ef_{T}\Big\|_{L^{p}(X)}^{p}\lessapprox\Big\|\sum_{T\in\mathbb{T}}\sum_{J\in\mathcal{J}_{\lambda}(T)}\mathcal Ef_{T}\mathbf{1}_{J}\Big\|_{L^{p}(X_{1})}^{p}.$$

Furthermore, we partition $\mathbb{T}=\bigcup_{h}\mathbb{T}_{h}$, where $h \in[1,R^{\varepsilon^{2}}]$ is a dyadic number such that for all $T\in\mathbb{T}_{h}$, we have
$\#\mathcal{J}_{\lambda}(T)\approx h$. Hence, we have
$$\sum_{T\in\mathbb{T}}\sum_{J\in\mathcal{J}_{\lambda}(T)}\mathcal Ef_{T}\mathbf{1}_{J}=\sum_{h}\sum_{T\in\mathbb{T}_{h}}\sum_{J\in\mathcal{J}_{\lambda}(T)}\mathcal Ef_{T}\mathbf{1}_{J}.$$
For any $R^\frac12$-cube $Q\subset X_{1}$, by pigeonholing, there is a dyadic number $h(Q)$ such that
$$\int_{Q}\Big|\sum_{T\in\mathbb{T}}\sum_{J\in\mathcal{J}_{\lambda}(T)}\mathcal Ef_{T}\mathbf{1}_{J}\Big|^{p}\lessapprox \int_{Q}\Big|\sum_{T\in\mathbb{T}_{h(Q)}}\sum_{J\in\mathcal{J}_{\lambda}(T)}\mathcal Ef_{T}\mathbf{1}_{J}\Big|^{p}.$$
Since $\|\sum_{T\in\mathbb{T}}\sum_{J\in\mathcal{J}_{\lambda}(T)}\mathcal Ef_{T}\mathbf{1}_{J}\|_{L^{p}(Q)}$ is comparable up to a constant multiple for all $Q\subset X_{1}$, we use dyadic pigeonholing again to obtain a uniform dyadic number $h$ and a refinement $X_{2}\subset X_{1}$ such that the following properties
\begin{itemize}
    \item $h(Q)=h$,
    \item $\|\sum_{T\in\mathbb{T}_{h}}\sum_{J\in\mathcal{J}_{\lambda}(T)}\mathcal Ef_{T}\mathbf{1}_{J}\|_{L^{p}(Q)}$ is comparable up to a constant multiple,
\end{itemize}
hold for any $Q\subset X_{2}$.
Consequently, we conclude
\begin{equation}\label{eq reduction2}
\Big\|\sum_{T\in\mathbb{T}}\mathcal Ef_{T}\Big\|_{L^{p}(X)}^{p}\lessapprox \int_{X_{2}}\Big|\sum_{T\in\mathbb{T}_{h}}\sum_{J\in\mathcal{J}_{\lambda}(T)}\mathcal Ef_{T}\mathbf{1}_{J}\Big|^{p}.
\end{equation}
Moreover, we have $|X_{2}|\gtrapprox|X|$ since $X_{2}$ is a refinement of $X$.

\subsection{The non-two-ends scenario}

Suppose $h \leq R^{\varepsilon^{4}}$. This corresponds to the non-two-ends scenario, and the result can be established directly via induction on scales as in \cite{WangWu2024}.

\subsection{Two-ends scenario}
Suppose $h \in [R^{\varepsilon^{4}},R^{\varepsilon^{2}}]$. For each $T\in \mathbb{T}_{h}$, considering the shading
$$Y(T)=\bigcup_{J\in\mathcal{J}_{\lambda}(T)} J\cap X. $$
It is straightforward to verify that the shading $(\mathbb{T}_{h},Y)$ is $(\varepsilon^{2},\varepsilon^{4})$ two-ends and $\lambda h$-dense.
After applying a $1/R$-dilation, the tubes have radius $R^{-1/2+\varepsilon_{0}}$. We set $\delta = R^{-1/2+\varepsilon_{0}}$.

On the one hand, we take 
$$\mu=R^{2\varepsilon^{2}}m^{\frac{9}{10}}(\lambda h)^{-\frac{101}{100}}R^{\frac{49}{100}}(R^{-\frac{3}{2}}\#\mathbb{T})^{\frac{1}{10}}.$$
Note that the radius of each $T$ is $R^{\frac{1}{2}+\varepsilon_{0}}$. We partition $T$ into $\lesssim R^{O(\varepsilon_{0})}$ many tubes of radius $R^{1/2}$ and we can use the Proposition \ref{prop incidence}. Therefore, there exists a set $X_{3}\subset X$ with 
\begin{equation}\label{eq multi}
    \sup_{Q\subset X_{3}}\#\{T\in\mathbb{T}_{h}:Y(T)\cap Q \neq \emptyset\}\lessapprox R^{O(\varepsilon_{0})}\mu
\end{equation}
such that $|X\setminus X_{3}|\leq R^{-\varepsilon^{2}}|X|$ by the Proposition \ref{prop incidence}. Denote by $X_{4}=X_{2}\cap X_{3}$. Then $|X_{4}|\gtrapprox|X_{2}|\gtrapprox|X|$ since $X_{2}$
is a refinement of $X$. Moreover, since $\|\sum_{T\in\mathbb{T}_{h}}\sum_{J\in \mathcal{J}_{\lambda}(T)}\mathcal Ef_{T}\mathbf{1}_{J}\|_{L^{p}(Q)}$ is comparable up to a constant multiple for all $Q\subset X_{2}$,
\begin{equation}\label{eq reduction4}
    \int_{X_{2}}\Big|\sum_{T\in\mathbb{T}_{h}}\sum_{J\in \mathcal{J}_{\lambda}(T)}\mathcal Ef_{T}\mathbf{1}_{J}\Big|^{p}\lessapprox\int_{X_{4}}\Big|\sum_{T\in\mathbb{T}_{h}}\sum_{J\in \mathcal{J}_{\lambda}(T)}\mathcal Ef_{T}\mathbf{1}_{J}\Big|^{p}.
\end{equation}
Let $B_{k}$ be a family of finitely overlapping $R^{1-\varepsilon^{2}}$ balls covering $B_{R}$. For each $B_{k}$,  denote $$\mathbb{T}_{h,k}=\{T\in \mathbb{T}_{h}: \exists J\in \mathcal{J}_{\lambda}(T),J\cap B_{k}\neq \emptyset \}.$$ Therefore, we have
\begin{equation}\label{eq reduction3}
    \begin{aligned}
        &\int_{X_{4}}\Big|\sum_{T\in \mathbb{T}_{h}}\sum_{J\in \mathcal{J}_{\lambda}(T)}\mathcal Ef_{T}\mathbf{1}_{J}\Big|^{\frac{10}{3}}\lessapprox\sum_{k} \int_{X_{4}\cap B_{k}}\Big|\sum_{T\in \mathbb{T}_{h}}\sum_{J\in \mathcal{J}_{\lambda}(T)}\mathcal Ef_{T}\mathbf{1}_{J}\Big|^{\frac{10}{3}}\\
        & \lessapprox \sum_{k}\int_{X_{4}\cap B_{k}}\Big|\sum_{T\in \mathbb{T}_{h,k}}\mathcal Ef_{T}\Big|^{\frac{10}{3}}\lessapprox R^{O(\varepsilon^{2})}\sup_{k}\int_{X_{4}\cap B_{k}}\Big|\sum_{T\in \mathbb{T}_{h,k}}\mathcal Ef_{T}\Big|^{\frac{10}{3}}.
    \end{aligned}
\end{equation}
Note that 
$$\#\{T\in \mathbb{T}_{h}:Y(T)\cap Q\neq \emptyset\}=\#\{T\in \mathbb{T}_{h,k}:T\cap Q\neq \emptyset\}$$
for each $Q\subset X_{4}\cap B_{k}\subset X_{3}\cap B_{k}$.
Applying Theorem \ref{thm refineddecoupling} and inequality \eqref{eq multi}, we have

\begin{equation}\label{eq refineddecoupling}
    \int_{X_{4}\cap B_{k}}\Big|\sum_{T\in\mathbb{T}_{h,k}}\mathcal Ef_{T}\Big|^{\frac{10}{3}}\lessapprox R^{O(\varepsilon_{0})}\mu^{\frac{2}{3}}\sum_{T\in \mathbb{T}}\|\mathcal Ef_{T}\|_{L^{\frac{10}{3}}(\omega_{B_{R}})}^{\frac{10}{3}}.
\end{equation}

Since $\|\mathcal Ef_{T}\|_{L^{\frac{10}{3}}(\omega_{B_{R}})}$ is comparable up to a constant multiple for all $T\in \mathbb{T}$, $\|f_{T}\|_{2}$ is comparable up to a constant multiple for all $T\in \mathbb{T}$, and $\# \mathbb{T}(\theta)\approx m$ for all $\theta \in \Theta$, we have  
\begin{equation}\label{eq Hol}
\sum_{T\in \mathbb{T}(\theta)}\|\mathcal Ef_{T}\|_{L^\frac{10}{3}(\omega_{B_R})}^{\frac{10}{3}}\lesssim \sum_{T\in \mathbb{T}(\theta)}\|\mathcal Ef_{T}\|_{L^2(\omega_{B_R})}^{\frac{10}{3}}R^{-\frac{5}{3}}\lesssim m^{-\frac{2}{3}}\|\mathcal Ef_{\theta}\|_{2}^{\frac{10}{3}}R^{-\frac{5}{3}}\lesssim m^{-\frac{2}{3}}\|f_{\theta}\|_{2}^{\frac{10}{3}}
\end{equation}
where the inequalities follow from Lemma \ref{lem TomasStein} and the H\"older inequality.
Combining with the inequalities \eqref{eq refineddecoupling} and \eqref{eq Hol}, we have
\begin{equation}\label{eq standard}
    \begin{aligned}
         \int_{X_{4}\cap B_{k}}\Big|\sum_{T\in\mathbb{T}_{h,k}}\mathcal Ef_{T}\Big|^{\frac{10}{3}}\lessapprox R^{O(\varepsilon_{0})} (\frac{\mu}{m})^{\frac{2}{3}}R^{-1}\|f\|_{2}^{\frac{10}{3}}.
    \end{aligned}
\end{equation}
Consequently, by the definition of $\|f_{\theta}\|_{L^{2}_{avg}(\theta)}$ and inequality \eqref{eq reduction3}, we obtain the following.
\begin{equation}
    \int_{X_{4}}\Big|\sum_{T\in\mathbb{T}_{h}}\sum_{J\in\mathcal{J}_{\lambda}(T)}\mathcal Ef_{T}\mathbf{1}_{J}\Big|^{\frac{10}{3}}\lessapprox R^{O(\varepsilon_{0}+\varepsilon^{2})}(\mu m^{-1})^{\frac{2}{3}}R^{-1}\|f\|_{2}^{2}\sup_{\theta}\|f_{\theta}\|_{L^{2}_{avg}(\theta)}^{\frac{4}{3}}.
\end{equation}
Recall $\mu=R^{2\varepsilon^{2}}m^{\frac{9}{10}}(\lambda h)^{-\frac{101}{100}}R^{\frac{49}{100}}(R^{-\frac{3}{2}}\#\mathbb{T})^{\frac{1}{10}}$ and $m\geq 1$, then we have 
\begin{equation}\label{eq 10/3}
    \int_{X_{4}}\Big|\sum_{T\in\mathbb{T}_{h}}\sum_{J\in\mathcal{J}_{\lambda}(T)}\mathcal Ef_{T}\mathbf{1}_{J}\Big|^{\frac{10}{3}} \lessapprox R^{O(\varepsilon_{0}+\varepsilon^{2})}(\lambda R)^{-\frac{101}{150}}\|f\|_{2}^{2}\sup_{\theta}\|f_{\theta}\|_{L^{2}_{avg}(\theta)}^{\frac{4}{3}}.
\end{equation}
On the other hand, using $L^{2}$-orthogonality and recalling $h \leq R^{\varepsilon^{2}}$, we obtain the following 
\begin{equation}\label{eq ortho}
    \int_{X_{4}}\Big|\sum_{T\in\mathbb{T}_{h}}\sum_{J\in\mathcal{J}_{\lambda}(T)}\mathcal Ef_{T}\mathbf{1}_{J}\Big|^{2} \lesssim R^{\varepsilon^{2}}\lambda R\|f\|_{2}^{2}.
\end{equation}
Interpolating with the inequalities \eqref{eq 10/3} and \eqref{eq ortho}, then there exists a constant $C_{\varepsilon}$ such that
\begin{equation}
     \int_{X_{4}}\Big|\sum_{T\in\mathbb{T}_{h}}\sum_{J\in\mathcal{J}_{\lambda}(T)}\mathcal Ef_{T}\mathbf{1}_{J}\Big|^{p} \leq C_{\varepsilon} R^{\varepsilon}\|f\|_{2}^{2}\sup_{\theta}\|f_{\theta}\|_{L^{2}_{avg}(\theta)}^{p-2}
\end{equation}
holds for $$p\geq 2+\frac{200}{251}.$$
Combining the inequalities \eqref{eq reduction2} and \eqref{eq reduction4}, we complete the proof of Theorem \ref{thm st} in the two-ends scenario.

\begin{remark}
    The exponent $p\geq 2+\frac{200}{251}$ implies that the Hausdorff dimension of Kakeya sets is at least $3.02$, which is smaller than the lower bound of the Hausdorff dimension of Kakeya sets in \cite{katz21}. In particular, in Lemma \ref{lem incidence}, we sacrifice a power of $\delta$ to obtain a smaller power of $\lambda$, which ensures that the incidence estimate can be utilized in restriction estimates. Additionally, in the sense of maintaining a usable restriction estimate, the incidence estimate in Lemma \ref{lem incidence} can no longer use a self-improving argument to improve the power of $\delta$. Nevertheless, for the restriction estimate, the range $p\geq 2+\frac{200}{251}$ is an improvement over the previously best known range $p\geq 2+\frac{1407}{1759}$ in \cite{HickmanRogers}.
\end{remark}

\section{The proof of Theorem \ref{thm: Kakeya Maximal}}

By a standard reduction from the volume estimate to the Kakeya maximal estimate (see, e.g., \cite{Wolff95,Guth-Zahl,katz21}), it suffices to prove the following estimate, which can be viewed as a level set estimate for the left-hand side of \eqref{eq: Kakeya Maximal}, i.e.,
$
\Big\|\sum_{T\in\mathbb{T}}\chi_{T}\Big\|_{\frac{d}{d-1}}.
$

\begin{theorem}\label{thm shading}
If $d_{0}=\frac{159+\sqrt{145}}{56}$, then for every $\varepsilon>0$ there exists a constant $C_\varepsilon>0$ such that
    \begin{align}
    \left|
        \bigcup_{T\in\T}
        Y(T)
    \right|
    \geq C_\varepsilon
    \lambda^{d_{0}}
    \delta^{4-d_{0}+\varepsilon}
    (\delta^3\#\T),
\label{eq: TE for proof Kakeya Max}
\end{align}
where $\mathbb{T}$ is a family of $\delta$-tubes pointing in $\delta$-separated directions with a $\lambda$-dense 
shading $Y$.
\end{theorem}

The rest of this section is devoted to the proof of Theorem \ref{thm shading}. 
Our objective is to establish a $\mathbf{TE}(d,d,b)$-type estimate
\eqref{eq: TE for proof Kakeya Max}. 
This allows us greater freedom in the exponent of $\lambda$ 
compared to the 
$\mathbf{TE}\left(d,\frac{d+1}{2}, b\right)$-type estimate in Lemma \ref{lem incidence}, 
which we exploit by repeatedly applying a self-improvement argument to improve the order of $\delta$.

To improve the Kakeya maximal estimate, we use Lemma \ref{lem hairbrush}, which offers a good power of $\lambda$, and establish a self-improving mechanism following \cite[Lemma 7.1]{katz21}.

\begin{lemma}
\label{thm:self-improving}
    Suppose $\operatorname{{\bf TE}}(4-\alpha, \beta, 1-\frac\alpha3)$ holds for some 
    \[
    \alpha^*:=\frac{75-3\sqrt{145}}{40}
        \leq
        \alpha 
        \leq1
        \quad\text{and}\quad
        \dfrac{131}{60}
        \leq 
        \beta 
        \leq 
        \frac{65}{24}. 
    \]
    Then, the assertions
    $\operatorname{{\bf TE}}(4-\alpha', \beta, 1-\frac{\alpha}{3})$
    and
    $\operatorname{{\bf TE}}(4-\alpha'', 4-\alpha'', 1)$ 
    hold for
    \[
        \alpha'= 
        1 -
        \frac{(18-17\alpha)(3-2\alpha)}{54(2-\alpha)}
        < \alpha
        \quad \text{and} \quad
        \alpha''=\frac{45}{28}-\frac{9}{14\alpha}.
    \]
\end{lemma}

\begin{proof}

Let $0<\varepsilon_{2}<1$. Fix $(Y, \T)$ to be a family of $\delta$-tubes pointing in $\delta$-separated directions with an associated $\lambda$-dense, $(\varepsilon_1, \varepsilon_{1}\varepsilon_2)$-two-ends shading $Y$, for any $0<\varepsilon_{1}<1$. 
Let 
the refined family of $\lambda_{Y_{1}}$-densely shaded tubes $(\T_1, Y_1)$,
the scale 
$\rho\in[\delta,1]$, 
the set $\Omega\in S^3$ of $\rho$-separated directions, 
the family $\T_\rho$ of essentially distinct $\rho$-tubes, 
the shading $Y_\rho$ on $\T_\rho$, 
and the uniform multiplicity $\tilde\mu$ of $\delta$-tubes inside each $T_\rho\in\T_\rho$ be the outputs of 
{Lemma \ref{lem dich}}.

By the same steps \eqref{eq rho}-\eqref{eq relation} at the beginning of the proof of Lemma \ref{lem incidence}, we have the following estimate:
\begin{align}
\label{eq: mu_Y_{1} by mu_fine}
    \mu_{Y_{1}}
    \lessapprox
    \lambda_{Y_\rho}^{-\frac7{16}}
    \rho^{-\frac34}
    A^\frac38
    (\rho^3\# \T_{\rho})^{\frac18}
    \tilde \mu
    .
\end{align}
    Unlike the proof of Lemma
    \ref{lem incidence}, the shading here satisfies the $(\varepsilon_{1},\varepsilon_{1}\varepsilon_{2})$-two-ends condition for any $0<\varepsilon_{1}<1$, hence there is no $\delta^{-\varepsilon_1}$-loss.

By the assumption that $\operatorname{{\bf TE}}(4-\alpha,\beta, 1-\frac\alpha3)$ holds, 
the incidence $\tilde \mu$ of $\delta$-tubes inside each $T_\rho\in\T_\rho$ is controlled by 
\begin{equation}
    \tilde{\mu}
    \lessapprox  \lambda_{Y_{1}}^{1-\beta} \cdot
    \left( \frac{\delta}{\rho} \right)^{-\alpha} \cdot
    \left( \frac{\delta^{3} \#\mathbb{T}_{1}}{\rho^{3} \#\mathbb{T}_{\rho}} \right)^{\alpha/3}
    \lessapprox \lambda_{Y_{1}}^{1-\beta} \cdot
    \left( \frac{\#\mathbb{T}_{1}}{\#\mathbb{T}_{\rho}} \right)^{\alpha/3}.
    \label{eq: Y_fine_hypothesis}
\end{equation}
Plugging (\ref{eq: Y_fine_hypothesis}) in \eqref{eq: mu_Y_{1} by mu_fine}, this implies 
\begin{align}
\label{mu_Y_{1} by hypothesis}
    \mu_{Y_{1}}
    \lessapprox 
    \lambda_{Y_\rho}^{-\frac7{16}}
    \rho^{-\frac34}
    A^\frac38
    (\rho^3\# \T_{\rho})^{\frac18}
    \lambda_{Y_{1}}^{1-\beta}
    \left(
        \frac{\# \T_{1}}{\# \T_{\rho}}
    \right)^\frac\alpha3. 
\end{align}
Further interpolate this with the bound from 
Lemma \ref{lem dich} (property \ref{item dich ix}):
\begin{align}
\label{mu_Y_{1} by refinement lemma}
    \mu_{Y_{1}}
    \lessapprox
    \rho^{-1}\tilde \mu
    \lessapprox
    \rho^{-1}
    \lambda_{Y_{1}}^{1-\beta}
    \left(
        \frac{\# \T_{1}}{\# \T_{\rho}}
    \right)^\frac\alpha3
\end{align}
to remove the $\rho$-scaled parameter $A$ and $\# \T_{\rho}$:
\begin{align}
    \mu_{Y_{1}}
    & \lessapprox
    (\ref{mu_Y_{1} by hypothesis})^{\frac{2\alpha}3}
    (\ref{mu_Y_{1} by refinement lemma})^{
        1-\frac{2\alpha}3
    }
    \lessapprox 
    \lambda_{Y_{1}}^{-\frac{7\alpha}{24}+1-\beta}
    \rho^{-1+
    \frac{5
    \alpha}{12}}
    A^{\frac{\alpha}{4}}
    \# \T_{\rho}^{-\frac{\alpha}{4}}
        \# \T_{1}^\frac\alpha3
    \notag
    \\
    & \lessapprox 
    \lambda_{Y_{1}}^{-\frac{7\alpha}{24}+1-\beta}
    \rho^{-1+\frac{7\alpha}{6}}
    \delta^{-\frac{3\alpha}4}
        \# \T_{1}^\frac\alpha{12}
    \left[A
        \cdot
        \frac{\delta^3\# \T_{1}}{\rho^3\# \T_{\rho}}
    \right]^{\frac\alpha4}
    \notag
    \\
    & \lessapprox 
        \lambda_{Y_{1}}^{-\frac{7\alpha}{24}+1-\beta}
    \rho^{-1+\frac{7\alpha}{6}}
    \delta^{-\alpha}
    (\delta^3\# \T_{1})^\frac\alpha{12}
    .
    \label{mu_Y_{1} bounded with only THETA}
\end{align}
Finally, to cancel the scale $\rho$ itself, interpolate the previous display (\ref{mu_Y_{1} bounded with only THETA}) with Guth--Zahl's trilinear Kakeya estimate 
\begin{align}
\label{GuthZahl trilinear in self-impro proof}
    \mu_{Y_{1}} \lessapprox \lambda_{Y_{1}}^{-\frac94}
    \rho^{-1}
    \delta^{-\frac34}
    (\delta^3\# \T_{1})^{\frac34}
\end{align}
to get 
\begin{align}
\label{mu_Y_{1} alpha"}
    \mu_{Y_{1}} \lessapprox
    (\ref{mu_Y_{1} bounded with only THETA})^{\frac6{7\alpha}}
    (\ref{GuthZahl trilinear in self-impro proof})^{1-\frac6{7\alpha}}
    \lessapprox 
    \lambda_{Y_{1}}^{
        -\frac52
        +
        (1-\beta)\frac6{7\alpha}
        + \frac{27}{14\alpha}
    }
    \delta^{   
        -\alpha''
    }
    (\delta^3\# \T_{1})^{
        \frac{23}{28}-\frac9{14\alpha}
    }
    .
\end{align}
If 
    $\beta\leq\frac{65}{24} 
    \approx 2.7083
    $,    
the power of $\lambda_{Y_{1}}$ in (\ref{mu_Y_{1} alpha"}) becomes
\[
        -\frac52
        +
        (1-\beta)\frac6{7\alpha}
        + \frac{27}{14\alpha}
        \geq 
        \alpha''-3,
\]
so (\ref{mu_Y_{1} alpha"}) shows that $\operatorname{{\bf TE}}(4-\alpha'', 4-\alpha'', 1)$ is true.

In order to obtain the self-improving property $\operatorname{{\bf TE}}(4-\alpha', \beta, 1-\frac\alpha3)$, we further interpolate (\ref{mu_Y_{1} alpha"}) with 
Wang--Wu's two-ends estimate
\begin{align}
\label{WW 2ends Furst DELTA}
    \mu_{Y_{1}} \lessapprox 
    \lambda_{Y_{1}}^{-\frac34}\delta^{-1} 
    \left(
        \delta^3\# \T_{1}
    \right)^{\frac12}
\end{align} 
 at the scale $\delta$ to get
\begin{align}
\label{eq: mu_Y_{1} alpha'}
    \mu_{Y_{1}}
    \lessapprox 
    (\ref{mu_Y_{1} alpha"})^{t}
    (\ref{WW 2ends Furst DELTA})^{1-t}
    \lessapprox  
    \lambda_{Y_{1}}^{
        \frac{196 \alpha^2 + 96 \alpha \beta - 525 \alpha - 144 \beta + 306}
        {108(2 - \alpha)}
    }
    \delta^{
        -\alpha'
    }
    (\delta^3\# \T_{1})^{\frac\alpha3},
\end{align}
where $t:=\frac{14\alpha(3-2\alpha)}{27(2-\alpha)}\in[0,1]$ is chosen to match the power of $(\delta^3\# \T)$ to 
$\frac\alpha3$. 
If 
    $
            \alpha^*
        \leq
        \alpha\leq1
    $ 
and {
    $\beta\geq 
   \frac{131}{60}
    $
},
the power of $\lambda_{Y_1}$ in 
\eqref{eq: mu_Y_{1} alpha'}
\[
    \frac{196 \alpha^2 + 96 \alpha \beta - 525 \alpha - 144 \beta + 306}
    {108(2 - \alpha)}
    \geq 1-\beta,
\]
so \eqref{eq: mu_Y_{1} alpha'} shows that $\operatorname{{\bf TE}}(4-\alpha', \beta, 1-\frac\alpha3)$ is true.

\end{proof}

Now we establish the following lemma by iterating the above self-improving lemma.  
\begin{lemma}\label{lem incidence2}
    Let $\alpha^*:=\frac{75-3\sqrt{145}}{40}$ and let 
    $
        d_0 = 
        \frac{159 + \sqrt{145}}{56}
        \approx 3.0543
    $.
    For any $\varepsilon>0$, the assertions
    \begin{align}
    \label{eq: iteration results}
        \operatorname{{\bf TE}}
        \left(4-\alpha^*-\varepsilon,\, 
        \frac{65}{28}, \,
        1-\frac{\alpha^{*}}{3}
        \right)
        \quad
        \text{and}
        \quad
        \operatorname{{\bf TE}}
        \left(
             {d_0-\varepsilon}, \,
             d_0
            {-\varepsilon}
            , 
            \, 
            1
        \right)
    \end{align}
    are true. 
\end{lemma}

\begin{proof}
    Fix $\alpha_1=1$. Wolff's hairbrush argument implies that assertion $\operatorname{{\bf TE}}(3,2,\frac12)$ is true. Since $\delta^{3}\#\T \lesssim 1$, the estimate $\operatorname{{\bf TE}}(4-\alpha_1, \frac{65}{28}, 1-\frac{\alpha_1}{3})$ also holds.

    Here $\beta=\frac{65}{28}$ is chosen just to satisfy the assumption needed in 
    {Lemma \ref{thm:self-improving}}, which allows us to iterate this self-improving argument. 

    Suppose that $\alpha^*\leq \alpha_k \leq 1$ and that the assertion $\operatorname{{\bf TE}}(4-\alpha_k, \frac{65}{28}, 1-\frac{\alpha_k}{3})$ is true. Define 
    \[
        \alpha_{k+1}:=
        1 -
        \frac{(18-17\alpha_k)(3-2\alpha_k)}{54(2-\alpha_k)}
        \in
        [\alpha^*, \alpha_k],
    \]
    then {Lemma \ref{thm:self-improving}} implies that the assertion $\operatorname{{\bf TE}}(4-\alpha_{k+1}, \frac{65}{28}, 1-\frac{\alpha_{k}}{3})$ is true and the assertion $\operatorname{{\bf TE}}(4-\alpha_{k+1}, \frac{65}{28}, 1-\frac{\alpha_{k+1}}{3})$ follows since $\alpha_{k+1}\leq\alpha_k$. 
    As $\alpha_k\searrow\alpha^*$, this iteration continues until some $\alpha_K<\alpha^*+\varepsilon$ and proves the first assertion in (\ref{eq: iteration results}). The second assertion in \eqref{eq: iteration results} follows immediately from the first and 
    {Lemma \ref{thm:self-improving}}. 
\end{proof}

Applying the standard two-ends reduction (e.g. \cite{Wolff95,Tao2009TwoEnds,katz21}), 
we see that the estimate 
$\operatorname{\mathbf{TE}}(d_0-\varepsilon, d_0-\varepsilon, 1)$ implies that Theorem \ref{thm shading} is true, and thus complete the proof of Theorem \ref{thm: Kakeya Maximal}.

\begin{remark}
   In the self-improving argument, the incidence estimate in Lemma \ref{lem hairbrush} only decreases the power of $\lambda$ without altering the dimensional exponent of $\delta$, compared to the two-ends hairbrush estimate and the X-ray estimate used in \cite{katz21}. This is precisely why we are able to improve the Kakeya maximal estimate, but not the lower bound on the Hausdorff dimension of Kakeya sets.

    On the other hand, the X-ray estimate~\cite{LT_X_ray} used in~\cite{katz21} more effectively captures the behavior of tubes that are essentially distinct but possibly almost parallel. In this regime, it is stronger---at the same dimensional exponent in~$\delta$---than 
    the two-ends hairbrush estimate 
    when handling the $A$-parallel tubes at scale~$\rho$, although it does incur an unfavorable dependence on the density parameter~$\lambda$. 
    This
    is precisely why~\cite{katz21} applies the X-ray estimate at scale~$\rho$ to obtain the improved lower bound $3.059$ for the Hausdorff dimension of Kakeya sets in $\mathbb{R}^4$.
   
\end{remark}

\bibliographystyle{alpha}
\bibliography{reference}

\end{document}